\documentclass{amsart}
\usepackage[dvipdfm]{graphicx}
\usepackage[all]{xy}
\title [On holomorphic sections of Veech holomorphic families] {On holomorphic sections of Veech holomorphic families of Riemann surfaces}

\author{Yoshihiko Shinomiya}
\address{Department of Mathematics
Tokyo Institute of Technology
2-12-1 Ookayama, Meguro-ku, Tokyo 152-8551, JAPAN}
\email{shinomiya.y.aa@m.titech.ac.jp}
\thanks{This work was supported by JSPS KAKENHI Grant Number 24005650.}
\begin{document}
\maketitle
\bibliographystyle{alpha} 

\theoremstyle{plain}
\newtheorem{theorem}{Theorem}[section]

\theoremstyle{definition}
\newtheorem{definition}[theorem]{Definition}

\theoremstyle{plain}
\newtheorem{proposition}[theorem]{Proposition}

\theoremstyle{plain}
\newtheorem{lemma}[theorem]{Lemma}

\theoremstyle{plain}
\newtheorem{corollary}[theorem]{Corollary}

\theoremstyle{definition}
\newtheorem{example}[theorem]{Example}

\theoremstyle{remark}
\newtheorem*{remark}{\bf Remark}

\theoremstyle{remark}
\newtheorem*{problem}{\bf Problem}
\begin{abstract}
We give upper bounds of the numbers of holomorphic sections of Veech holomorphic families of Riemann surfaces. 
The numbers depend only on the topological types of base Riemann surfaces and fibers.
We also show a relation between types of Veech groups and moduli of cylinder decompositions of flat surfaces.
\end{abstract}

\section{Introduction}
Let $M(g, n)$ be the moduli space of Riemann surfaces of type $(g, n)$ with $3g-3+n>0$.
A triple $(M, \pi, B)$ of a two-dimensional complex manifold $M$, a Riemann surface $B$, and a holomorphic map $\pi: M \rightarrow B$ is called a holomorphic family of Riemann surfaces of type $(g, n)$ over $B$ if each fiber $X_t=\pi^{-1}(t)$ is a Riemann surface of type $(g, n)$ and the induced map $B \ni t \mapsto X_t \in M(g, n)$ is holomorphic. 
The Riemann surface $B$ is called the base space of the holomorphic family  $(M, \pi, B)$ of Riemann surfaces.
A holomorphic family of Riemann surfaces is called locally non-trivial if the induced map is non-constant.
Imayoshi and Shiga \cite{ImaShi88} proved that every locally non-trivial holomorphic family $(M, \pi ,B)$ of Riemann surfaces of finite type has only finitely many holomorphic sections if  the base space $B$ is of finite type.
Shiga \cite{Shiga97} gave upper bounds of the numbers of holomorphic families of Riemann surfaces of type $(g, n)= (0, n)  (n \geq 4)$, $(1, 2)$ and $(2, 0)$.
The upper bounds also give upper bounds of the numbers of holomorphic sections of holomorphic families of Riemann surfaces of such types. 

In this paper, we study holomorphic sections of {\it Veech holomorphic families of Riemann surfaces}. 
Let $X$ be a Riemann surface of type $(g, n)$ with $3g-3+n>0$.
A flat structure $u$ is an Euclidean structure on $X$ with finitely many singularities such that every transition function is of the form $w=\pm z+c$.  
The pair $(X, u)$ of a Riemann surface $X$ and a flat structure $u$ on $X$ is called a flat surface.
All punctures of $X$ and singularities of $u$ are called critical points of the flat surface $(X, u)$.
Denote by $C(X, u)$ the set of all critical points of $(X, u)$.
We assume that the Euclidean areas of flat surfaces are finite.
The affine group ${\rm Aff}^+(X, u)$ of a flat surface $(X, u)$ is the group of all quasiconformal self-maps of $X$ which preserve $C(X, u)$ and are affine with respect to the flat structure $u$. 
Each element $h$ of ${\rm Aff}^+(X, u)$ is called an affine map.
The derivatives $A \in {\rm GL}(2, \mathbb{R})$ of the descriptions $w=Az+c$ of an affine map $h$ is uniquely determined up to the sign and they are in ${\rm SL}(2, \mathbb{R})$.
Thus, we have a homomorphism $D: {\rm Aff}^+(X, u) \rightarrow {\rm PSL}(2, \mathbb{R})$.
We call the homomorphism the derivative map.
The Veech group $\Gamma(X, u)$ of $(X, u)$ is the image of the derivative map.
Veech \cite{Veech89} proved that $\Gamma(X, u)$ is a Fuchsian group and the mirror image $\overline{B}$ of the orbifold $\mathbb{H}/\Gamma(X, u)$ is holomorphically and locally isometrically embedded into the moduli space $M(g, n)$ with the Teichm\"uller metric. 
 Let $B$ be a Riemann surface obtained from $\overline{B}$ by removing all cone points. 
 The Veech holomorphic family of Riemann surfaces of type $(g, n)$ over $B$ induced by $(X, u)$ is the holomorphic family of Riemann surfaces corresponding to the holomorphic 
embedding of $B$ into the moduli space $M(g, n)$.
 We show in \cite{Shinomiya12vhf} that every holomorphic section of Veech holomorphic families of Riemann surfaces is locally the orbit of a point $a \in (X, u)$ for Teichm\"uller deformations and the point satisfies ${\rm Aff}^+(X, u)\{a\}= {\rm Ker}(D)\{a\}$.
In \cite{Shinomiya12vhf}, we also give upper bounds of the numbers of holomorphic sections of Veech holomorphic families of Riemann surfaces such that the corresponding flat surfaces have simple Jenkins-Strebel directions.
The upper bounds depend only on the topological types of fibers and base spaces.
However, flat surfaces do not have simple Jenkins-Strebel direction in general.
In this paper, we give upper bounds of the numbers of holomorphic sections of all Veech holomorphic families of Riemann surfaces which depend only on the topological types of fibers and base spaces.  
We also give a relation between types of Veech groups and moduli of cylinder decompositions of flat surfaces by  Jenkins-Strebel directions.
It claims that ratios of moduli of cylinders restrict to the types of Veech groups.

\section{Preliminaries}\label{Preliminaries}
In this section, we  define flat surfaces, Veech groups, and Veech holomorphic families of Riemann surfaces.
We also study their properties and some theorems in \cite{Shinomiya12vhf} which are referred to in this paper.

Let $\overline{X}$ be a compact Riemann surface of genus $g$ and $X$ a Riemann surface of type $(g, n)$ with $3g-3+n>0$ which is $\overline{X}$ with $n$ points removed.
\begin{definition}[Flat structure and flat surface]
A flat structure $u$ on $X$ is an atlas of $X$ with finitely many singular points which satisfies the following conditions: 
\begin{enumerate} 
\item[(1)] local coordinates of $u$ are compatible with the orientation of $X$,
\item[(2)] transition functions are the form 
\begin{eqnarray}
w=\pm z+c \nonumber
\end{eqnarray}
 in $z(U \cap V)$ for $(U, z), (V, w) \in u$ with $U \cap V \not = \phi$,
\item[(3)]  the atlas $u$ is maximal with respect to $(1)$ and $(2)$.
\end{enumerate}
A pair $(X, u)$ of a Riemann surface $X$ and a flat structure $u$ on $X$ is called a flat surface.
All punctures of $X$ and singular points of $u$ are called critical points of $(X, u)$.
The set of all critical points is denoted by $C(X, u)$.
\end{definition}
On a flat surface $(X, u)$, we may consider Euclidean geometry.
For instance, area, segments, lengths or directions of the segments are considered on $(X, u)$. 
In this paper, we assume that the Euclidean area of $(X, u)$ is finite.
A $\theta$-closed geodesic on $(X, u)$ is a closed geodesic with direction $\theta \in [0, \pi)$ which does not contain critical points.
A segment connecting critical points with direction $\theta$ is called a $\theta$-saddle connection.  
If $\theta=0$, we also call a $\theta$-closed geodesic a horizontal closed geodesic, and a $\theta$-saddle connection a horizontal saddle connection.
\begin{definition}[Jenkins-Strebel direction] 
A direction $\theta \in [0, \pi)$ is called a Jenkins-Strebel direction if almost all points in $(X, u)$ are contained in
$\theta$-closed geodesics. 
\end{definition}
Let $\theta \in [0, \pi)$ be a Jenkins-Strebel direction.
If $z \in (X, u)$ is not contained in $\theta$-closed geodesics, then $z$ is on a $\theta$-saddle connection.
Let us remove all $\theta$-saddle connections from $X$.
Then $X$ has finitely many connected components.
Every connected component $R$ is a cylinder foliated by $\theta$-closed geodesics and the boundary of $R$ consists of saddle connections.
The core curves of the cylinders are not homotopic to each other and not homotopic to a point or a puncture. See \cite{Strebel84}.
\begin{definition}
If a Jenkins-Strebel direction $\theta$ decomposes $X$ into $m$ cylinders $R_1, \cdots, R_m$, then the direction $\theta$ is called a $m$-Jenkins-Strebel direction.
In particular, if $m=1$, then we call the direction $\theta$ a simple Jenkins-Strebel direction. 
The decomposition 
$\{R_i\}$
is called a cylinder decomposition of $(X, u)$ by a Jenkins-Strebel direction $\theta$.
\end{definition}
\begin{remark}
Since the core curves of $R_i$'s are not homotopic to each other, $m$ is less than $3g-3+n$.
\end{remark}

Let $u=\{(U, z)\}$ be a flat structure on $X$.
For every $A \in {\rm SL}(2, \mathbb{R})$, a flat surface $A\cdot (X, u)= (X, u_A)$ is defined by
$u_A=\{(U, A\circ z)\}$.
The set $C(X, u_A)$ of  critical points of $(X, u_A)$ coincides with $C(X, u)$ and $u_A$ gives a new complex structure  of $X$.
The ${\rm SL}(2, \mathbb{R})$-orbit of the flat surface $(X, u)$  in the Teichm\"uller space $T(X)$ is defined by $\Delta=\{[A\cdot(X, u), id]: A \in {\rm SL}(2, \mathbb{R})\}$.
It is easy to show that $[A\cdot(X, u), id]=[UA\cdot(X, u), id]$ in $T(X)$ for $A \in {\rm SL}(2, \mathbb{R})$, $U \in {\rm SO}(2)$.
Thus, the bijection $\phi :  {\rm SO}(2)\setminus {\rm SL}(2, \mathbb{R}) \rightarrow  \mathbb{H}$ defined by  $\phi({\rm SO}(2)\cdot A)=-\overline{A^{-1}(i)}$ induces a map $\iota : \mathbb{H} \rightarrow T(X)$.
Here, $A^{-1}(\cdot)$ acts on $\mathbb{H}$ as a M\"obius transformation.
\begin{proposition}[\cite{EarGar97},\cite{HerSch07}]
The map $\iota: \mathbb{H} \rightarrow T(X)$ is holomorphic and isometric embedding from the hyperbolic plane $\mathbb{H}$ into the Teichm\"uller space $T(X)$ with the Teichm\"uller distance.
 Every holomorphic and isometric embedding from $\mathbb{H}$ into $T(X)$ is constructed from a flat surface as above.
\end{proposition}
Let $\iota: \mathbb{H} \rightarrow T(X)$ be a holomorphic and isometric embedding constructed from a flat surface $(X, u)$ of type $(g, n)$.
The image $\Delta=\iota(\mathbb{H})$ 
is called a Teichm\"uller disk.
We consider the image of the Teichm\"uller disk $\Delta$ into the moduli space $M(g, n)$.
Since $M(g, n)=T(X)/{\rm Mod}(X)$, the image of the Teichm\"uller disk is described as $\Delta/{\rm Stab}(\Delta)$. 
Here, ${\rm Mod}(X)$ is the mapping class group of $X$ and ${\rm Stab}(\Delta)$ is the subgroup of ${\rm Mod}(X)$ consisting of all mapping classes which preserve $\Delta$. 
\begin{definition}[Affine groups]
A quasiconformal self-map $h$ of $X$ is called an affine map of $(X, u)$ if $h$ preserves $C(X, u)$ and, for $(U, z)$ and $(V, w) \in u$ with $h(U) \subset V$, $w \circ h \circ z^{-1}$ is the from
$w=Az+c$
for some $A \in {\rm GL}(2, \mathbb{R})$ and $c \in \mathbb{C}$.
The group of all affine maps of $(X, u)$ is denoted by ${\rm Aff}^+(X, u)$ and we call it the affine group of $(X, u)$.
\end{definition}
By the definition of flat structures, the derivative $A$ of the affine map $w \circ h \circ z^{-1}$ is uniquely determined up to the sign. 
Moreover, the assumption that the area of $(X, u)$ is finite implies that $A$ is in ${\rm SL}(2, \mathbb{R})$.
Therefore, we obtain a homomorphism $D: {\rm Aff}^+(X, u)\rightarrow {\rm PSL}(2, \mathbb{R})$.
The homomorphism is called the derivative map.
\begin{definition}[Veech groups]
We call the image $\Gamma(X, u)=D({\rm Aff}^+(X, u))$ of the derivative map $D$ the Veech group of $(X, u)$.
\end{definition}
\begin{proposition}[\cite{Veech89}]
Affine maps of $(X, u)$ are not homotopic to each other. 
Hence, the group ${\rm Aff}^+(X, u)$ is considered as a subgroup of ${\rm Mod}(X)$.
\end{proposition}
Veech proved that Veech groups give the images of Teichm\"uller disks into the moduli spaces as follows.
\begin{theorem}[\cite{Veech89}, \cite{EarGar97}, \cite{HerSch07} ]
The affine group ${\rm Aff}^+(X, u)$ coincides with ${\rm Stab}(\Delta)$.
For $t \in \mathbb{H}$ and $h \in {\rm Aff}^+(X, u)$, we have $h_*(\iota(t))=\iota(RAR^{-1}(t))$.
Here, 
$A=D(h)$, 
$
R=\tiny{\left(
  \begin{array}{cccc}
   -1 & 0  \\
    0 & 1
  \end{array}
  \right)}$ and
$RAR^{-1}$ is a M\"obius transformation which acts on $\mathbb{H}$. 
\end{theorem}

\begin{corollary}[\cite{Veech89}, \cite{EarGar97}, \cite{HerSch07} ]
The Veech group $\Gamma(X, u)$ is a Fuchsian group. 
Let $\overline{\Gamma}(X, u)=R\Gamma(X, u)R^{-1}$.
The orbifold $\mathbb{H}/\overline{\Gamma}(X, u)$ is holomorphically and locally isometrically embedded into the moduli space $M(g, n)$. 
The embedded orbifold is the image of the Teichm\"uller disk $\Delta$ into the moduli space $M(g, n)$. 
\end{corollary}
Veech holomorphic families of Riemann surfaces are given by such holomorphic and locally isometric embeddings $\Phi: \mathbb{H}/\overline{\Gamma}(X, u) \rightarrow M(g, n)$.
Let $\mathbb{H}^*$ be $\mathbb{H}$ with elliptic fixed points of $\overline{\Gamma}(X, u)$ removed. 
Since every point $t \in B=\mathbb{H}^*/\overline{\Gamma}(X, u)$ corresponds to a Riemann surface $X_t= \Phi(t)$,  we may construct a two-dimensional complex manifold $M$ by
\begin{eqnarray}
M=\left\{(t, z) : t \in B, z \in X_t=\Phi(t)\right\}. \nonumber
\end{eqnarray}
Let $\pi: M \rightarrow B$ be the projection $\pi(t, z)=t$. 
Then the triple $(M, \pi, B)$ is a holomorphic family of Riemann surfaces.
See \cite{Shinomiya12vhf}, for more details of the construction of the holomorphic families of Riemann surfaces.
\begin{definition}[Veech holomorphic families of Riemann surfaces]
A holomorphic family of Riemann surfaces constructed as above is called a Veech holomorphic family of Riemann surfaces of type $(g, n)$ over $B$.

\end{definition}

Through this paper, we assume that the Veech groups $\Gamma(X, u)$ of flat surfaces $(X, u)$ are Fuchsian groups of finite types.
If the Veech group $\Gamma(X, u)$  of a flat surface $(X, u)$ is of type $(p, k ; \nu_1, \cdots , \nu_k)$ ($\nu_i \in \{2,3, \cdots , \infty \}$), 
the base space $B$ of the corresponding Veech holomorphic family of Riemann surfaces is of type $(p, k)$.

In \cite{Shinomiya12vhf}, we characterize holomorphic sections of Veech holomorphic families of Riemann surfaces.
Let $(X, u)$ be a flat surface of type $(g, n)$ and $(M, \pi, B)$ the Veech holomorphic family of Riemann surfaces defined by $(X, u)$.
A holomorphic section of $(M, \pi, B)$ is a holomorphic map $s: B \rightarrow M$ such that $\pi \circ s= id_B$.
Let $\widetilde{\rho}: \mathbb{H} \rightarrow \mathbb{H}/\overline{\Gamma}(X, u)$ be the universal covering map.
For each $\widetilde{t} \in \mathbb{H}$, let 
$f_{\widetilde{t}} : X \rightarrow f_{\widetilde{t}}(X)$ be the Teichm\"uller map whose Beltrami coefficient $\mu_{\widetilde{t}}$ satisfies $\mu_{\widetilde{t}}=\frac{i-\widetilde{t}}{i+\widetilde{t}}\frac{d \bar{z}}{dz}$ for all $(U, z) \in u$. 
Then the Riemann surface $f_{\widetilde t}(X)$ coincides with $X_{\rho(\widetilde t)}=\pi^{-1}(\rho\left(\widetilde {t}\right))$. 
\begin{theorem}[\cite{Shinomiya12vhf}]\label{characterization}
Let $s: B \rightarrow M$ be a holomorphic section of the Veech holomorphic family $(M, \pi, B)$ of Riemann surfaces induced by $(X, u)$.
There exists $a \in (X, u)$ such that $s\circ \widetilde{\rho}(\widetilde{t})=(\widetilde{t}, f_{\widetilde{t}}(a))$ for all $\widetilde{t} \in \mathbb{H}^*$.
Moreover, the point $a \in (X, u)$ satisfies ${\rm Aff}^+(X, u)\{a \}= {\rm Ker}(D)\{ a\}$.
\end{theorem}
Let $\Gamma$ be a finite index subgroup of $\Gamma(X, u)$ and $\overline{\Gamma}=R\Gamma R^{-1}$.
Let $\rho: \mathbb{H}/\overline{\Gamma} \rightarrow \mathbb{H}/\overline{\Gamma}(X, u)$ be the covering map.
The holomorphic and locally isometric map $\Phi \circ \rho : \mathbb{H}/\overline{\Gamma}  \rightarrow M(g, n)$ constructs a holomorphic family of Riemann surfaces. 
Let $\mathbb{H}_{\overline{\Gamma}}^*$ be the upper half-plane $\mathbb{H}$ with elliptic fixed points of $\overline{\Gamma}$ removed. 
We set $B^\prime=\mathbb{H}_{\overline{\Gamma}}^*/\overline{\Gamma}$, $M^\prime= \left\{(t, z): t \in B^\prime, z \in X_t= \Phi \circ \rho(t)\right\}$, and $\pi^\prime: M \rightarrow B^\prime$ to be a projection $\pi^\prime(t, z)=t$.
Then the triple $(M^\prime, \pi^\prime, B^\prime)$ is a holomorphic family of Riemann surfaces of type $(g, n)$ over $B^\prime$. 

\begin{corollary}[\cite{Shinomiya12vhf}]
A holomorphic section $s^\prime: B^\prime \rightarrow M^\prime$ of the holomorphic family $(M^\prime, \pi^\prime, B^\prime)$ of Riemann surfaces as above  corresponds to a point $a \in X$ satisfying $D^{-1}(\Gamma)\{a\}={\rm Ker}(D)\{a\}$.
\end{corollary}
In \cite{Shinomiya12vhf}, we estimate the number of points $a \in X$ satisfying ${\rm Aff}^+(X, u)\{a \}= {\rm Ker}(D)\{ a\}$ in case that $(X, u)$ has a simple Jenkins-Strebel direction.
The estimation gives upper bounds of the numbers of holomorphic sections of Veech holomorphic families of Riemann surfaces induced by such flat surfaces.
\begin{theorem}[\cite{Shinomiya12vhf}] \label{upperbound1}
Let $(X, u)$ be a flat surface of type $(g, n)$ with a simple Jenkins-Strebel direction.
Let $(M, \pi, B)$ be the Veech holomorphic family of Riemann surfaces induced by $(X,u)$.
Suppose that the base space $B$ is of type $(p, k)$.
Then the number of holomorphic sections of  $(M, \pi, B)$ is at most
\begin{eqnarray}
32\pi (2p-2+k)(3g-3+n)^2(3g-2+n)-2g+2. \nonumber
\end{eqnarray}
\end{theorem} 
In Section \ref{main}, we give upper bounds of the numbers of holomorphic sections of all Veech holomorphic families of Riemann surfaces by extending the proof of Theorem \ref{upperbound1} in \cite{Shinomiya12vhf}.
We also apply the following theorem in \cite{Shinomiya12vhf}.
\begin{theorem} [\cite{Shinomiya12vhf}]\label{fuchsian}
Let $\Gamma < {\rm PSL}(2, \mathbb{R})$ be a Fuchsian group of type $(p,k: \nu_1, \cdots, \nu_k)$
$(\nu_i \in \{2, 3, \cdots, \infty\})$.
Let $k_0$ be the number of $\nu_i$'s which are equal to $\infty$.
Assume that  $\Gamma$ contains $ \tiny \left[\left(
  \begin{array}{cccc}
    1 & 1 \\
    0 & 1
  \end{array}
  \right)
  \right]$ and it is primitive.
Then there exists  
  $ \tiny \left[\left(
  \begin{array}{cccc}
    a & b \\
    c & d
  \end{array}
  \right)
  \right] \in \Gamma$
  such that 
\begin{equation}
\displaystyle
1\leq |c| <{\rm Area}(\mathbb{H}/\Gamma)-k_0+1.
\nonumber
\end{equation}
Here, ${\rm Area}\left(\mathbb{H}/\Gamma(X, u) \right) $ is the hyperbolic area of  the orbifold $\mathbb{H}/\Gamma$.
\end{theorem}
The first inequality is the consequence of the Shimizu lemma (c.f. \cite{ImaTan92}).
\begin{lemma}[The Shimizu lemma]
Let $\Gamma < {\rm PSL}(2, \mathbb{R})$ be a Fuchsian group.
Assume that
$ \tiny \left[\left(
  \begin{array}{cccc}
    1 & 1 \\
    0 & 1
  \end{array}
  \right)
  \right] \in \Gamma$ and it is primitive. 
 Then $c=0$ or $|c| \geq 1$  for all  $ \tiny \left[\left(
  \begin{array}{cccc}
    a & b \\
    c & d
  \end{array}
  \right)
  \right] \in \Gamma$. 
\end{lemma}

\begin{remark}
It is known that 
\begin{eqnarray}
 {\rm Area}\left(\mathbb{H}/\Gamma(X, u) \right) =2\pi\left(2p-2+\sum_{i=1}^k(1-\frac{1}{\nu_i})\right) \nonumber
\end{eqnarray}
for a Fuchsian group $\Gamma$ of type $(p, k: \nu_1, \cdots, \nu_k)$.
See \cite{FarKra92}.
The Shimizu lemma also means that  the horodisks centered at punctures of $\mathbb{H}/\Gamma$ whose areas are 1 do not intersect each other.
\end{remark}

\begin{proof}[Proof of Theorem \ref{fuchsian}]
Suppose that $|c| \geq {\rm Area}(\mathbb{H}/\Gamma)-k_0+1$ for all  $ \tiny \left[\left(
  \begin{array}{cccc}
    a & b \\
    c & d
  \end{array}
  \right)
  \right] \in \Gamma$ with $c \not =0$. 
  Let $p_1, \cdots, p_{k_0}$ be punctures of $\mathbb{H}/\Gamma$.
We assume that $p_1$ corresponds to $ \tiny \left[\left(
  \begin{array}{cccc}
    1 & 1 \\
    0 & 1
  \end{array}
  \right)
  \right] $.
Let $U_i$ be the horodisks centered at $p_i$ ($i \in \{2, \cdots, k_0 \}$) whose areas are $1$.
By the Shimizu lemma,  $U_i \cap U_j = \phi$ $(i \not = j)$.
By considering the Ford region of $\Gamma$, it is easily proved that $p_1$ has the horodisk $U_1$ whose area is ${\rm Area}(\mathbb{H}/\Gamma)-k_0+1$ and
$U_1 \cap U_i= \phi$ for all $i \in \{2, \cdots, k_0 \}$. 
Moreover, $\mathbb{H}/\Gamma - \bigcup_{i=1}^{k_0} U_i$ has a positive area.
Therefore, 
\begin{eqnarray}
\displaystyle {\rm Area}(\mathbb{H}/\Gamma) > \sum_{i=1}^{k_0} {\rm Area}(U_i) ={\rm Area}(\mathbb{H}/\Gamma). \nonumber
\end{eqnarray}
This is a contradiction.
\end{proof}
Finally, we see another property of Veech groups. 
\begin{theorem}[The Veech dichotomy theorem \cite{Veech89}]\label{dichotomy}
Let $(X, u)$ be a flat surface.
Suppose that the Veech group  $\Gamma(X,u)$ of $(X, u)$ is of finite type.
Then every direction $\theta \in [0, \pi)$ satisfies one of the following properties:
\begin{itemize}
\item The direction $\theta$ is a Jenkins-Strebel direction.
Let $\{R_i \}_{i=1}^m$ be the cylinder decomposition of $(X, u)$ by the direction $\theta$. 
Then the ratio ${\rm mod}(R_i)/{\rm mod}(R_j)$ is a rational number for all $i,j \in \{1, \cdots, k\}$. 
\item Every $\theta$-direction geodesic is dense in $X$ and uniquely ergodic. 
That is, the $\theta$-direction geodesic flow has only one transverse measure $\mu$ up to scalar multiples such that the flow is ergodic with respect to $\mu$.
\end{itemize}
Here, the modulus ${\rm mod}(R_i)$ of the cylinder $R_i$ is the ratio of the circumference to the height.
\end{theorem}
For details of ergodicity,
 see \cite{KatHas96}, \cite{Nikolaev01}.

\section{Main Theorems}\label{main}
In this section, we prove the following two theorems.
First theorem gives upper bounds of numbers of holomorphic sections of all Veech holomorphic families of Riemann surfaces. 
The second theorem gives a relation between types of Veech groups of flat surfaces and moduli of cylinders of cylinder  decompositions of the flat surfaces.
\begin{theorem}\label{mainthm1}
Let $(M, \pi, B)$ be a Veech holomorphic family of Riemann surfaces of type $(g, n)$ over $B$.
Suppose that the base space $B$ is a Riemann surface of type $(p, k)$. 
Then the number of holomorphic sections of $(M, \pi, B)$ is at most
\begin{eqnarray}
32 \pi (2p-2+k) (3 g-3+n)^2  \left\{ 2(3 g-3+n) + 3 \exp \left(\frac{5}{e}(3 g-3+n)\right) \right\}.\nonumber
\end{eqnarray}
\end{theorem}

\begin{theorem}\label{mainthm2}
Let $(X, u)$ be a flat surface of type $(g, n)$.
Suppose that the Veech group $\Gamma(X, u)$ is of type $(p, k : \nu_1, \cdots, \nu_k )$ $(\nu_i \in \{2, 3, \cdots, \infty\})$.
Let $\{ R_i \}$ be the cylinder decomposition of $(X, u)$  by a Jenkins-Strebel direction.
Then, 
\begin{eqnarray}
 \left( \frac{{\rm mod}(R_i)}{{\rm mod}(R_j)} \right)^{\frac{1}{2}}< 4\pi\exp \left(\frac{5}{e}(3g-3+n)\right)\left(2p-2+\sum_{r=1}^k\left(1- \frac{1}{\nu_r}\right)\right) \nonumber
\end{eqnarray}
for all $i, j \in \{1, \cdots, m \}$.
\end{theorem}
By Theorem \ref{characterization}, a proof of Theorem \ref{mainthm1} is given by estimating the cardinality of the set 
\begin{eqnarray}
S(X, u)=\left\{z \in X : {\rm Aff}^+(X, u)\{z \}= {\rm Ker}(D)\{ z\} \right\}. \nonumber
\end{eqnarray}
Let $\varphi : X \rightarrow Y=X/{\rm Ker}(D)$ be the quotient map.
Then $Y$ has a flat structure $u^\prime$ induced by the flat structure $u$ on $X$.
Assume that $Y$ is of type $(g^\prime, n^\prime)$.
Since $\Gamma(X, u)={\rm Aff}^+(X, u)/{\rm Ker}(D)$, we may consider $\Gamma(X, u)$ as a subgroup of the affine group ${\rm Aff}^+(Y, u^\prime)$.
Then, $\varphi(S(X, u))=\{w \in Y: \Gamma(X, u)\{w\}=\{w\}\}$.
Theorem \ref{dichotomy} and  the assumption that $\Gamma(X, u)$ is of finite type imply that the set of all Jenkins-Strebel directions of $(X, u)$ is dense in $[0, \pi)$.
We may assume that the direction $\theta=0$ is a $m$-Jenkins-Strebel direction of $(X, u)$ for some $1\leq m\leq 3g-3+n$.
Then $\theta=0$ is also a $m^\prime$-Jenkins-Strebel direction of $(Y, u^\prime)$ for some $m^\prime \leq m$.
Let $\{R_i\}_{i=1}^m$ and $\{R_j^\prime\}_{j=1}^{m^\prime}$ be the cylinder decompositions of $(X, u)$ and $(Y, u^\prime)$ by the direction $\theta=0$, respectively.
We define a map $\sigma: \{1, \cdots, m \} \rightarrow \{1, \cdots, m^\prime\}$ by $\sigma(i)=j$ if $\varphi(R_i)=R_j^\prime$.
By Theorem \ref{dichotomy}, the Veech group $\Gamma(X, u)$ contains elements of the form 
$\tiny \left[\left(
  \begin{array}{cccc}
    1 & b \\
    0 & 1
  \end{array}
  \right)
  \right]$ $(b \not =0)$.
As $\Gamma(X, u)$ is of finite type, $\Gamma(X, u)$ also contains elements of the form 
$\tiny \left[\left(
  \begin{array}{cccc}
    * & * \\
    c & *
  \end{array}
  \right)
  \right]$ $(c \not =0)$.
We set
\begin{eqnarray}
b_0=\inf \left\{ |b| : \tiny \left[\left(
  \begin{array}{cccc}
    1 & b \\
    0 & 1
  \end{array}
  \right)
  \right] \in \Gamma(X, u), b \not =0 \right\} \nonumber 
\end{eqnarray}
and
\begin{eqnarray}
c_1=\inf \left\{|c| : \tiny \left[\left(
  \begin{array}{cccc}
    a & b \\
    c & d
  \end{array}
  \right)
  \right] \in \Gamma(X, u), c \not =0 \right\}. \nonumber
\end{eqnarray}
Let $A=\tiny \left[\left(
  \begin{array}{cccc}
    a_1 & b_1 \\
    c_1 & d_1
  \end{array}
  \right)
  \right] $ and $B=\tiny \left[\left(
  \begin{array}{cccc}
    1 & b_0 \\
    0 & 1
  \end{array}
  \right)
  \right]$  be elements of $\Gamma(X, u)$ which attain the numbers $c_1$ and $b_0$, respectively. 
Denote by $h_A^\prime$ and $h_B^\prime$ the elements of ${\rm Aff}^+(Y, u^\prime)$ corresponding to $A$ and $B $, respectively.
Let $I_j^\prime$ be a vertical open interval connecting two boundary components $L_{j,1}^\prime$ and $L_{j,2}^\prime$ of $R_j^\prime$ for each $j \in \{1, \cdots, m^\prime \}$. 
Let $l_w^\prime$ be the horizontal closed geodesic containing $w \in I_j^\prime$.
We set 
\begin{center}
$I_{j, 0}^\prime =\{ w \in I_j^\prime : l_w^\prime$ contains a fixed point of $h_B^\prime \}$.
\end{center}
Then we consider the set
\begin{eqnarray}
 {\rm Cross}(A)=\left(\bigcup_{j=1}^{m^\prime}\bigcup_{w \in I_{j, 0}^\prime}\left( l_w^\prime \cap h_A^\prime\left(l_w^\prime \right)\right)\right)
 \bigcup 
\left(\bigcup_{j=1}^{m^\prime}\bigcup_{r=1,2} \left(   L_{j,r}^\prime \cap h_A^\prime(L_{j,r}^\prime)\right)\right). \nonumber
\end{eqnarray}
\begin{lemma}\label{lemma1}
The set ${\rm Cross}(A)$ contains $\varphi(S(X,u))$.
\end{lemma}
\begin{proof}
Since $\varphi \left(S(X, u)\right)$ is the set of all fixed points of $\Gamma(X, u)$ on $(Y, u^\prime)$, the set $\varphi \left(S(X, u)\right)$ is contained in ${\rm Cross}(A)$.
\end{proof}
Let $H_i, H_j^\prime$ be heights of the cylinders $R_i, R_j^\prime$ and   $W_i, W_j^\prime$ circumferences of the cylinders $R_i, R_j^\prime$, respectively.
Recall that the moduli of $R_i$ and $R_j^\prime$ are defined by ${\rm mod}(R_i)=W_i/H_i$ and ${\rm mod}(R_j^\prime)=W_j^\prime/H_j^\prime$.
Since 
\begin{eqnarray}
W_i/\sharp {\rm Ker}(D)  \leq &W_{\sigma(i)}^\prime &\leq W_i \nonumber
\end{eqnarray}
and
\begin{eqnarray}
H_i/2 \leq & H_{\sigma(i)}^\prime &\leq H_i,  \nonumber
\end{eqnarray}
we have
\begin{eqnarray}
{\rm mod}(R_i)/\sharp {\rm Ker}(D) \leq  {\rm mod}(R_{\sigma(i)}^\prime) \leq 2{\rm mod}(R_i). \nonumber
\end{eqnarray}
\begin{lemma}\label{lemma2}
If $j=\sigma(i)$, then 
\begin{eqnarray}
\sharp I_{j,0}^\prime \leq \left\lceil \frac{b_0 \sharp {\rm Ker}(D)}{{\rm mod}(R_i)}\right\rceil. \nonumber
\end{eqnarray}
Here, $\left\lceil x \right\rceil$ is the smallest integer which is greater than or equal to $x$.
\end{lemma}
\begin{proof}
Suppose that $j=\sigma(i)$. 
If $h_B^\prime$ does not fix $R_j^\prime$, then $\sharp I_{j, 0}^\prime=0$.
If $h_B^\prime$ fixes $R_j^\prime$ and permutes two boundary components of $R_j^\prime$, then $\sharp I_{j, 0} = 1$.
Otherwise, by identifying $R_j^\prime$ with $[0, W_j^\prime)\times (0, H_j^\prime)$, the affine map $h_B^\prime$ can be represented as
\begin{eqnarray}
h_B^\prime\left(
  \begin{array}{cccc}
    u \\
    v
  \end{array}
  \right)= \left(
  \begin{array}{cccc}
    1 & b_0 \\
    0 & 1
  \end{array}
  \right)
  \left(
  \begin{array}{cccc}
    u \\
    v
  \end{array}
  \right)
 + \left(
  \begin{array}{cccc}
    W_j^\prime \xi  \\
    0
  \end{array}
  \right) \nonumber
\end{eqnarray} 
for some $0 \leq \xi <1$.
Thus, 
$\sharp I_{j,0}^\prime  \leq \left\lceil \frac{b_0 H_j^\prime}{W_j^\prime}+1\right\rceil-1 \leq \left\lceil \frac{b_0 \sharp {\rm Ker}(D)}{{\rm mod}(R_i)}\right\rceil$.
\end{proof}
\begin{lemma}\label{lemma3}
For every $j \in \{1, \cdots , m^\prime\}$ and $w \in I_{j, 0}^\prime$, the inequality
\begin{eqnarray}
\sharp \left( l_w^\prime \cap h_A^\prime(l_w^\prime)\right) \leq 2 {\rm mod}(R_i) c_1 \nonumber
\end{eqnarray}
holds if $j=\sigma(i)$.
\end{lemma}
\begin{proof}
Assume that $j=\sigma(i)$. 
Since $A {\tiny \left(
  \begin{array}{cccc}
    W_j^\prime\\
    0 
  \end{array}
  \right)}=
   {\tiny   \left(
  \begin{array}{cccc}
   W_j^\prime a_1\\
    W_j^\prime c_1
  \end{array}
  \right)}$, 
  for any $w \in I_{j,0}^\prime$,
  \begin{eqnarray}
\sharp \left( l_w^\prime \cap h_A^\prime(l_w^\prime)\right) \leq W_j^\prime c_1/H_j^\prime \leq 2 {\rm mod}(R_i) c_1. \nonumber
\end{eqnarray} 
\end{proof}
\begin{lemma}\label{lemma4}
\begin{eqnarray}
\displaystyle \sharp
\bigcup_{j=1}^{m^\prime}\bigcup_{r=1,2} \left(   L_{j,r}^\prime \cap h_A^\prime(L_{j,r}^\prime)\right) \leq \sum_{i=1}^m 4 {\rm mod}(R_i)c_1.\nonumber
\end{eqnarray}
\end{lemma}
\begin{proof}
\begin{eqnarray}
\sharp \bigcup_{j=1}^{m^\prime}\bigcup_{r=1,2} \left(   L_{j,r}^\prime \cap h_A^\prime(L_{j,r}^\prime)\right) 
&\leq& \sum_{j=1}^{m^\prime} \sum_{r=1,2} \sharp \left(   L_{j,r}^\prime \cap h_A^\prime(L_{j,r}^\prime)\right) \nonumber\\
&\leq & \sum_{j=1}^{m^\prime} 2 
\frac{W_j^\prime c_1}{H_j^\prime} 
W_j^\prime c_1/H_j^\prime \nonumber \\
&\leq & \sum_{i=1}^m 4 {\rm mod}(R_i)c_1. \nonumber
\end{eqnarray}
\end{proof}
\begin{proposition}\label{prop}
\begin{eqnarray}
\displaystyle \sharp S(X, u) \leq  
2 b_0 c_1\sharp{\rm Ker}(D)\sum_{i=1}^m \left( \sharp{\rm Ker}(D) + \frac{3{\rm mod}(R_i)}{b_0}\right) \nonumber
\end{eqnarray}
\end{proposition}
\begin{proof}
By Lemma \ref{lemma2}, Lemma \ref{lemma3} and Lemma \ref{lemma4}, we have
\begin{eqnarray}
\sharp S(X, u)&\leq& \sharp \varphi^{-1}({\rm Cross}(A)) \nonumber\\
&\leq & \sharp{\rm Ker}(D)\cdot  \sharp {\rm Cross}(A) \nonumber\\
& \leq & \sharp {\rm Ker}(D)\sum_{i=1}^m 2{\rm mod}(R_i) c_1 \left( \left\lceil \frac{b_0 \sharp {\rm Ker}(D)}{{\rm mod}(R_i)}\right\rceil+2 \right)\nonumber\\
& \leq & \sharp {\rm Ker}(D)\sum_{i=1}^m 2{\rm mod}(R_i) c_1 \left( \frac{b_0 \sharp {\rm Ker}(D)}{{\rm mod}(R_i)} +3 \right)\nonumber\\
&=& 2 b_0 c_1\sharp{\rm Ker}(D)\sum_{i=1}^m \left( \sharp{\rm Ker}(D) + \frac{3{\rm mod}(R_i)}{b_0}\right). \nonumber
\end{eqnarray}
\end{proof}
To prove Theorem \ref{mainthm1} and \ref{mainthm2}, we estimate $b_0c_1$, $\sharp {\rm Ker}(D)$, $\frac{{\rm mod}(R_i)}{b_0}$, and $\frac{b_0}{{\rm mod}(R_i)}$.
By Theorem \ref{fuchsian}, we have
\begin{eqnarray}
b_0c_1 \leq {\rm Area}\left(\mathbb{H}/\Gamma(X, u) \right). \nonumber
\end{eqnarray}
For all $i \in \{1, \cdots, m\}$, let $s_i^1, s_i^2$ be the numbers of horizontal saddle connections on two boundary components of $R_i$. 
 
\begin{lemma}\label{segments}
\begin{eqnarray}
\displaystyle \sum_{i=1}^{m}\left(s_i^1+s_i^2\right) \leq 4(3 g-3+n).
\nonumber
\end{eqnarray}
\end{lemma}
\begin{proof}
By considering the Euler characteristic of $X$, the equation
\begin{eqnarray}
\displaystyle \sum_{i=1}^{m}\left(s_i^1+s_i^2\right) =2(2g-2+\sharp C(X, u)) \nonumber
\end{eqnarray}
holds.
Let $q$ be the holomorphic quadratic differential on $(X, u)$ such that $q=dz^2$ for $(U, z)\in u$.
All points of $X \setminus C(X, u)$ are not zeros of $q$.
Every point of $C(X, u) \cap X$ is a zero of $q$ and the orders of the punctures of $X$ with respect to $q$ are greater than or equal to $-1$. 
By the Riemann-Roch theorem, the sum of orders of all zeros of $q$ is $4g-4$.
Hence, $\sharp C(X, u) \leq 4g-4+2n$ and we obtain the claim.
\end{proof}
\begin{lemma}\label{ker}
\begin{eqnarray}
\sharp {\rm Ker}(D) \leq 4(3 g-3+n).\nonumber
\end{eqnarray}
\end{lemma}
\begin{proof}
Each element of ${\rm Ker}(D)$ is uniquely determined by the image of a saddle connection and its orientation if it exists.
Therefore, we obtain $\displaystyle \sharp {\rm Ker}(D)\leq \sum_{i=1}^{m}\left(s_i^1+s_i^2\right) \leq 4(3g-3+n)$.
\end{proof}
\begin{lemma}\label{ratio}
For all $i \in \{1, \cdots, m\}$, 
\begin{eqnarray}
\frac{1}{2}\exp \left(-\frac{5}{e}(3 g-3+n)\right) < \frac{b_0}{{\rm mod}(R_i)}< 2\exp\left(\frac{5}{e}(3 g-3+n)\right){\rm Area}\left(\mathbb{H}/\Gamma(X, u) \right) ^2. \nonumber 
\end{eqnarray}
\end{lemma}
\begin{remark}
 Theorem \ref{mainthm2} is immediately proved from Lemma \ref{ratio}.
\end{remark}
To prove Lemma \ref{ratio}, we consider the Landau function $G(n)$. 
The Landau function $G(n)$ is the greatest order of an element of the symmetric group $S_n$ of degree $n$.
\begin{remark}
Landau showed that 
\begin{eqnarray}
\lim_{n \rightarrow \infty}\frac{\log(G(n))}{\sqrt{n \log(n)}}=1\nonumber
\end{eqnarray}
 and
Massias \cite{Massias84} showed that
\begin{eqnarray}
\displaystyle \log(G(n))\leq 1.05313\ldots \sqrt{n \log(n)} \nonumber
\end{eqnarray}
with equality at $n=1319766$. 
In this paper, we apply the following inequality which is easily proved
\begin{eqnarray}
\displaystyle G(n)\leq \exp \left(\frac{n}{e} \right). \nonumber
\end{eqnarray}
This inequality is better than Massias's if $n \leq 27$.
\end{remark}
\begin{proof}[Proof of Lemma \ref{ratio}]
Take $h_A, h_B \in {\rm Aff}^+(X, u)$ with $D(h_A)=A, D(h_B)=B$.
There exists $m_0 \leq G(m)$ such that $h_B^{m_0}(R_i)=R_i$ for all $i \in \{1, \cdots, m \}$.
Then $h_B^{2m_0}$ preserves each boundary component of $R_i$.
For each $i$, we choose $\alpha_i^r \leq s_i ^r$ $(r=1,2)$ such that  $h_B^{2m_0 \alpha_i^1 \alpha_i^2}$ fixes each boundary component of $R_i$ pointwise.
Set
\begin{eqnarray}
\displaystyle \alpha=2 m_0 \prod_{i=1}^m \alpha_i^1 \alpha_i^2 . \nonumber 
\end{eqnarray}
By Lemma \ref{segments}, 
\begin{eqnarray}
\alpha &\leq & \displaystyle 2 G(m) \prod_{i=1}^m s_i^1 s_i^2 \nonumber \\
&\leq & 2 \exp \left(\frac{m}{e} \right)\left\{ \sum_{i=1}^m \left( s_i^1+s_i^2 \right)/2 m \right\}^{2 m} \nonumber \\
&\leq  &2 \exp \left(\frac{1}{e}(3 g-3+n) \right)\left\{ \frac{2(3 g-3+n)}{m} \right\}^{2 m}  \nonumber\\
&< & 2 \exp \left(\frac{5}{e} (3 g-3+n)\right). \nonumber
\end{eqnarray}
Let $C_i$ be a core curve of the cylinder $R_i$ for each $i \in \{1, \cdots, m \}$.
Then $h_B^\alpha$ is a composition of right hand Dehn twists along $C_i$ $(i=1, \cdots, m)$.
Hence, for every $i \in \{1, \cdots, m \}$, there exists  $n_i \geq 1$ such that $\alpha b_0= n_i {\rm mod}(R_i)$.
Thus
\begin{eqnarray}
\frac{{\rm mod}(R_i)}{b_0}=\frac{\alpha}{n_i} < 2 \exp \left(\frac{5}{e} (3 g-3+n)\right). \nonumber
\end{eqnarray}
Next, let us consider the affine map $h=h_A^{-1}\circ h_B^\alpha \circ h_A$.
The derivative of $h$ is
\begin{eqnarray}
D(h)= A^{-1}BA=\left[\left(
  \begin{array}{cccc}
    1+\alpha b_0 c_1 d_1 & \alpha b_0 d_1^2 \\
    -\alpha b_0 c_1^2 & 1-\alpha b_0 c_1 d_1
  \end{array}
  \right)
  \right] .\nonumber
\end{eqnarray}
 Since $D(h) { \left(
  \begin{array}{cccc}
    W_i \\
    0
  \end{array}
  \right)}=
   {   \left(
  \begin{array}{cccc}
   W_i(1+\alpha b_0 c_1 d_1) \\
   - W_i \alpha b_0 c_1^2
  \end{array}
  \right)}$, 
we have
\begin{eqnarray}
\displaystyle W_i\alpha b_0 c_1^2= \sum_{j=1}^m \sharp \left(h(C_i)\cap C_j \right) H_j \nonumber
\end{eqnarray}  
for all $i \in \{1, \cdots, m \}$.
Then we have 
\begin{eqnarray}
{\rm mod}(R_i)\alpha b_0 c_1^2 &=& \frac{1}{H_i} \sum_{j=1}^m \sharp \left(h(C_i)\cap C_j \right) H_j \nonumber\\
&\geq& \sharp \left(h(C_i)\cap C_i \right) \nonumber\\
&=& \sharp \Bigl( h_B^\alpha \bigl(h_A\left(C_i\right)\bigr) \cap h_A\left(C_i\right) \Bigr)  \nonumber\\
&\geq & 1 \nonumber
\end{eqnarray}
since $h_A\left(C_i\right)$ intersects at least one of $C_j$'s and $h_B^\alpha$
fixes all boundary components of $R_j$'s.
Therefore, by Theorem \ref{fuchsian}, 
\begin{eqnarray}
\frac{b_0}{{\rm mod}(R_i)} &=&\frac{\left(b_0 c_1 \right)^2}{{\rm mod}(R_i) b_0 c_1^2} \nonumber\\
&<& 2 \exp \left(\frac{5}{e} (3 g-3+n)\right) {\rm Area}\left(\mathbb{H}/\Gamma(X, u) \right) ^2. \nonumber
\end{eqnarray}
\end{proof}
\begin{proof}[Proof of Theorem \ref{mainthm1}]
Since $m$ is the number of cylinders of a cylinder decomposition of $(X, u)$, we have $m \leq 3g-3+n$.
By 
Theorem \ref{fuchsian}, Proposition \ref{prop}, Lemma \ref{ker}, and Lemma \ref{ratio}, we have
\begin{eqnarray}
\displaystyle 
&&\sharp S(X, u) \nonumber\\
&\leq & 
2 b_0 c_1\sharp{\rm Ker}(D)\sum_{i=1}^m \left( \sharp{\rm Ker}(D) + \frac{3{\rm mod}(R_i)}{b_0}\right) \nonumber\\
&< & 32 \pi (2p-2+k) (3 g-3+n)^2  \left\{ 2(3 g-3+n) + 3 \exp \left(\frac{5}{e}(3 g-3+n)\right) \right\}.\nonumber
\end{eqnarray}
\end{proof}
Let $\Gamma$ be a finite index subgroup of $\Gamma(X, u)$ which is of type $(p^\prime, k^\prime : \nu_1^\prime, \cdots, \nu_{k^\prime}^\prime )$ $(\nu_i^\prime \in \{2, 3, \cdots, \infty\})$.
Let $(M^\prime, \pi^\prime, B^\prime)$ be a holomorphic family of Riemann surfaces corresponding to $\Gamma$ as in Section \ref{Preliminaries}.
By the same argument as above, we obtain the following corollary.
\begin{corollary}
The number of holomorphic sections of the holomorphic family $(M^\prime, \pi^\prime, B^\prime)$ of Riemann surfaces is at most
\begin{eqnarray}
 32 \pi (2p^\prime-2+k^\prime) (3 g-3+n)^2  \left\{ 2(3 g-3+n) + 3 \exp \left(\frac{5}{e}(3 g-3+n)\right)\right\}. \nonumber 
\end{eqnarray}
\end{corollary}


\section*{Acknowledgments}
The author would like to thank Hiroshige Shiga for introducing me to this research area.

\nocite{GutHubSch03}
\nocite{Imayoshi09}
\nocite{Moller06}
\nocite{Veech91}
\bibliography{ref}

\end{document}